\documentclass[11pt]{article}

\usepackage[letterpaper,margin=1in]{geometry}
\usepackage[T1]{fontenc}
\usepackage[utf8]{inputenc}
\usepackage{amsmath,amssymb,amsthm,mathtools}
\usepackage{newpxtext}
\usepackage{newpxmath}
\usepackage{microtype}
\usepackage{tikz}
\usepackage{pgfplots}
\usepackage{float}
\pgfplotsset{compat=1.18}
\usepackage[dvipsnames]{xcolor}
\usepackage{aliascnt}
\usepackage[hypertexnames=false]{hyperref}
\usepackage[nameinlink,noabbrev]{cleveref}

\hypersetup{
colorlinks=true, linkcolor=NavyBlue, citecolor=ForestGreen, urlcolor=NavyBlue, pdftitle={Extremal Spanning Trees in Product Grid Graphs}, pdfauthor={Jiechen Zhang} }

\numberwithin{equation}{section}

\newtheorem{theorem}{Theorem}[section]

\newaliascnt{corollary}{theorem}
\newtheorem{corollary}[corollary]{Corollary}
\aliascntresetthe{corollary}

\newaliascnt{conjecture}{theorem}
\newtheorem{conjecture}[conjecture]{Conjecture}
\aliascntresetthe{conjecture}

\newaliascnt{lemma}{theorem}
\newtheorem{lemma}[lemma]{Lemma}
\aliascntresetthe{lemma}

\newaliascnt{proposition}{theorem}
\newtheorem{proposition}[proposition]{Proposition}
\aliascntresetthe{proposition}

\theoremstyle{definition}
\newaliascnt{definition}{theorem}
\newtheorem{definition}[definition]{Definition}
\aliascntresetthe{definition}

\theoremstyle{remark}
\newaliascnt{remark}{theorem}
\newtheorem{remark}[remark]{Remark}
\aliascntresetthe{remark}

\crefname{theorem}{Theorem}{Theorems}
\crefname{corollary}{Corollary}{Corollaries}
\crefname{conjecture}{Conjecture}{Conjectures}
\crefname{lemma}{Lemma}{Lemmas}
\crefname{proposition}{Proposition}{Propositions}
\crefname{definition}{Definition}{Definitions}
\crefname{remark}{Remark}{Remarks}
\Crefname{theorem}{Theorem}{Theorems}
\Crefname{corollary}{Corollary}{Corollaries}
\Crefname{conjecture}{Conjecture}{Conjectures}
\Crefname{lemma}{Lemma}{Lemmas}
\Crefname{proposition}{Proposition}{Propositions}
\Crefname{definition}{Definition}{Definitions}
\Crefname{remark}{Remark}{Remarks}

\DeclareMathOperator{\arcosh}{arcosh}
\DeclareMathOperator{\arsinh}{arsinh}
\DeclareMathOperator{\Spec}{Spec}
\DeclareMathOperator{\dist}{dist}
\DeclareMathOperator{\Tr}{Tr}
\newcommand{\Z}{\mathbb{Z}}
\newcommand{\R}{\mathbb{R}}
\newcommand{\dd}{\,d}
\newcommand{\Path}{P}
\newcommand{\Cycle}{C}

\title{\bfseries Extremal Spanning Trees in Product Grid Graphs}
\author{Jiechen Zhang\footnote{EPFL\@. Email: \href{mailto:jiechen.zhang@epfl.ch}{jiechen.zhang@epfl.ch}.}}
\date{}

\begin{document}
\maketitle

\begin{abstract}
We compare spanning-tree counts at a fixed number of vertices for Cartesian products of paths $\Path_r$, $r\ge1$, or cycles $\Cycle_s$, $s\ge3$. With the number $d$ of factors fixed, an equal-volume tuple that is more balanced in the logarithmic majorization order has at least as many spanning trees, with equality only for permutations. This comparison persists after adjoining any fixed finite connected Cartesian factor; equal-sided blocks therefore uniquely maximize the count at admissible block volume $n^d$. We prove the comparison through strict log-convexity of one-dimensional heat traces in logarithmic length. The Hartman--Watson representation expresses the cycle interpolation as a positive mixture of theta functions, and the path case follows from it. The same heat-trace order decreases every positive inverse-power sum of the nonzero Laplacian eigenvalues, and hence the total effective resistance. Cylinders behave differently: square balancing can fail. For $\Path_r\square\Cycle_s$ of area $N=rs$, extending the product formula to real height and optimizing over integer circumferences $3\le s\le N$ yields an optimum on the scale $N^{1/3}$. Along areas having a divisor asymptotic to that optimum, quantitative estimates place every graph-maximizing circumference on the same scale. For two-dimensional tori, balancing remains optimal, but the gain in the logarithmic spanning-tree count admits no positive uniform lower bound proportional to the decrease in the side sum.
\end{abstract}

\section{Introduction}

Let $\tau(G)$ denote the number of spanning trees of a finite connected graph $G$. We compare product grids with the same number of vertices, varying their side lengths and boundary conditions. A path factor $\Path_r$ gives a free boundary; a cycle factor $\Cycle_s$ gives a periodic one. The two-dimensional products are therefore rectangles $\Path_r\square\Path_s$, tori $\Cycle_r\square\Cycle_s$, and cylinders $\Path_r\square\Cycle_s$. Path lengths are positive integers, and cycle lengths are at least $3$.

For rectangles, Zhang~\cite[Theorem~1.2]{Zhang2026Balancing} proves that the divisor pair closest to a square uniquely maximizes the count, up to rotation. The same comparison gives a sharp quantitative bound: balancing the sides at fixed area increases the logarithmic count by at least $\alpha=\arsinh(1)=\log(1+\sqrt2)$ times the decrease in their sum. Tori have the same maximizing shape. Their gains behave differently, however: strict balancing admits no positive universal linear side-sum coefficient.

For cylinders, the closest-to-square rule fails already at area $36$: $\tau(\Path_9\square\Cycle_4)>\tau(\Path_6\square\Cycle_6)$. The two directions have different roles: $r$ is the free height and $s$ the cyclic circumference, with area $N=rs$. In a long cylinder, the leading shape-dependent loss is $\alpha s+\pi N/(3s^2)$. The free-boundary cost increases with $s$, whereas the periodic deficit decreases, so their balance selects a circumference of order $N^{1/3}$ rather than the square scale $N^{1/2}$. A graph circumference must also divide $N$; the estimates in \Cref{sec:cylinders} control the graph optimum when a divisor lies near the analytic one.

In higher dimensions, logarithmic majorization~\cite{MarshallOlkinArnold2011} gives the relevant order. With the number of factors and the volume fixed, more balanced logarithmic side lengths yield more spanning trees in either the pure path or pure cycle family, with equality only for permutations. The equal-sided product follows as the unique maximizer at perfect-power volume. Integer pairwise balancing can stop at a nonuniform tuple even at such a volume; the global comparison still reaches the equal-sided product directly. At other volumes majorization is only a partial order and need not determine a unique maximizer.

We obtain this global comparison through heat traces, whose Cartesian-product factorization reduces the problem to strict log-convexity in logarithmic side length. Karamata's inequality then orders the traces, and integration orders the products of positive Laplacian eigenvalues. The pointwise order also controls inverse spectral moments, survives positive spectral shifts, and remains valid with an arbitrary fixed connected Cartesian factor. Within a mixed product, it balances the free factors among themselves or the periodic factors among themselves. The allocation of volume between those two types is a separate problem, as the cylinder already illustrates.

\subsection{Our Contributions}

The results fall into three families.

\begin{enumerate}
\setlength{\itemsep}{1pt}
\setlength{\parsep}{0pt}
\setlength{\topsep}{2pt}
\item[(1)] \emph{Logarithmic majorization for product grids} (\Cref{thm:log-majorization,thm:perfect-power-heat,cor:spectral-functionals}). We prove the strict comparison for a block of paths or cycles with any fixed finite connected Cartesian factor. At perfect-power block volume, equal-sided blocks uniquely maximize the spanning-tree count. The proof gives a pointwise heat-trace order, whose consequences extend beyond determinants: logarithmic balancing decreases completely monotone spectral sums, including all positive inverse moments. It also gives the determinant comparison for every positive Laplacian shift.
\item[(2)] \emph{Quantitative torus balancing} (\Cref{thm:torus-exact-gain,prop:torus-no-linear-constant}). An exact finite identity separates each nontrivial balancing gain into a nonnegative sampling term and a positive residual. The best universal nonnegative linear side-sum coefficient for torus gains is $0$, in contrast to the rectangle coefficient $\alpha$.
\item[(3)] \emph{Fixed-area cylinder optimization} (\Cref{thm:cylinder-real-height,thm:cylinder-divisor-rich}). Put $s_0=(2\pi N/(3\alpha))^{1/3}$. Allowing real height $r=N/s$, we locate every maximizing integer circumference over $3\le s\le N$ at $s_0+O(1)$. For graph cylinders, a divisor $u_N\mid N$, $u_N\ge3$, with $u_N/s_0\to1$ yields $|s_{\mathrm{opt}}-s_0|=O(|u_N-s_0|+1)$ for every maximizing circumference.
\end{enumerate}

\subsection{Related Work}

\paragraph{Relation to earlier work.}
For two path factors with no additional factor, \Cref{thm:log-majorization} recovers the qualitative part of \cite[Theorem~1.2]{Zhang2026Balancing}: balancing a rectangle at fixed area increases $\tau$, and the closest-to-square divisor pair is the unique maximizer up to rotation. The present argument does not quantify that gain. The sharp coefficient $\alpha$ belongs to the earlier quantitative comparison; \Cref{rem:rectangle-boundary} explains its boundary origin. The result of~\cite{Zhang2026SquareMaximality} has a different scope: the $n\times n$ square maximizes $\tau$ among all connected induced subgraphs of $\Z^2$ with $n^2$ vertices, including the equality characterization. It therefore contains the two-dimensional path case of \Cref{thm:perfect-power-heat} with no additional factor. Our comparison treats path or cycle blocks in every dimension, allows a fixed connected Cartesian factor, and yields the spectral consequences in \Cref{sec:spectral-consequences}. The induced-subgraph extension to $d\ge3$ is formulated in \Cref{conj:lattice-product-extremal}.

Exact enumeration supplies the common starting point. Kirchhoff's Matrix--Tree Theorem~\cite{Kirchhoff1847} turns product spectra into lattice spanning-tree formulas~\cite{Temperley1974,Wu1977,ShrockWu2000}; Tzeng and Wu give the hypercubic formulas for free, periodic, and mixed boundaries~\cite{TzengWu2000}. Other enumeration methods include the transfer matrices of Golin, Leung, Wang, and Yong for grids, cylinders, and tori~\cite{GolinLeungWangYong2005}, and Louis's formulas for circulant graphs and discrete tori~\cite{Louis2015}. We use the spectral products for finite shape comparisons.

Asymptotic results distinguish bulk growth from boundary and shape corrections. Chinta, Jorgenson, and Karlsson relate discrete-torus determinants to real tori~\cite{ChintaJorgensonKarlsson2010,ChintaJorgensonKarlsson2012}, while Louis allows the directions to grow at different rates~\cite{Louis2015Asymptotics}. The expansions of Izmailian and Kenna cover several boundary conditions~\cite{IzmailianKenna2015}, following~\cite{IzmailianKennaGuoWu2014}. Our bounded-aspect formulas recover these corrections after the Matrix--Tree normalization. To locate the cylinder optimum beyond the bounded-aspect regime, we combine uniform long-cylinder estimates with global bounds that exclude all other scales.

The heat-trace method also meets continuous shape optimization. Faulhuber~\cite{Faulhuber2021} studies the determinant maximum among unit-area rectangular flat tori, where the determinant is zeta regularized. Our determinants are finite products of graph eigenvalues. In the heat-trace proof, the Hartman--Watson transform~\cite{HartmanWatson1974,Yor1980} transfers theta-function convexity to the discrete interpolation.

For induced subgraphs, the planar square-maximality conjecture was recorded by Procaccia and Tucker-Foltz~\cite{ProcacciaTuckerFoltz2022}. Tapp's boundary-sensitive bounds~\cite{Tapp2024} give further context for leaving the product class. Arbitrary periodic identifications change the comparison class in another way: the counterexample in \Cref{sec:discussion} rules out the standard torus as a universal benchmark.

\paragraph{Organization.}
\Cref{sec:preliminaries} records the shared spectral identities. \Cref{sec:higher-dimensional} proves logarithmic majorization and its spectral and pairwise consequences. \Cref{sec:rectangles-tori} treats the quantitative rectangle and torus comparisons, and \Cref{sec:cylinders} proves the cylinder optimization results. \Cref{sec:discussion} describes the remaining extremal directions.

\section{Preliminaries}\label{sec:preliminaries}

We use the combinatorial Laplacian $L_G$ of a finite simple graph. For a tuple $\mathbf n=(n_1,\ldots,n_d)$, write $\Path_{\mathbf n}=\Path_{n_1}\square\cdots\square\Path_{n_d}$ and $\Cycle_{\mathbf n}=\Cycle_{n_1}\square\cdots\square\Cycle_{n_d}$. Each product has $n_1\cdots n_d$ vertices. The spectral normalization below is shared by the heat-trace argument and the exact formulas for the three two-dimensional boundary types.

\subsection{Spectra and one-dimensional products}

For $r\ge1$ define $\lambda_i^P(r)=2-2\cos(\pi i/r)$ for $0\le i\le r-1$, and for $r\ge3$ define $\lambda_i^C(r)=2-2\cos(2\pi i/r)$ for $0\le i\le r-1$.  These are the Laplacian spectra of $\Path_r$ and $\Cycle_r$, respectively; the spectrum of a Cartesian product consists of all sums of eigenvalues from its factors, with multiplicity. Kirchhoff's Matrix--Tree Theorem~\cite{Kirchhoff1847} gives
\[
  |V(G)|\,\tau(G)=\prod_{\lambda\in\Spec^+(G)}\lambda,
\]
where $\Spec^+(G)$ denotes the positive Laplacian eigenvalues counted with multiplicity.

\begin{lemma}\label{lem:oned-products}
For $r\ge1$, $\prod_{i=1}^{r-1}\lambda_i^P(r)=r$, with the empty product convention when $r=1$.  For $r\ge3$, $\prod_{i=1}^{r-1}\lambda_i^C(r)=r^2$.
\end{lemma}

\begin{proof}
A path has one spanning tree, including the one-vertex case, whereas a cycle has $r$, one for each deleted edge. The Matrix--Tree normalization gives the positive eigenvalue products $r$ and $r^2$, respectively.
\end{proof}

For the later spectral products, we use the hyperbolic form of the path and cycle characteristic polynomials.

\begin{samepage}
\begin{lemma}\label{lem:path-cycle-polynomials}
Let $z>0$ and $x=2\cosh z-2$.  Then
\[
  \prod_{i=1}^{r-1}\bigl(x+\lambda_i^P(r)\bigr)
  =
  \frac{\sinh(rz)}{\sinh z},
  \qquad r\ge1,
\]
and
\[
  \prod_{i=1}^{r-1}\bigl(x+\lambda_i^C(r)\bigr)
  =
  \left[
  \frac{\sinh(rz/2)}{\sinh(z/2)}
  \right]^2,
  \qquad r\ge3.
\]
Moreover,
\[
  \prod_{i=0}^{r-1}\bigl(x+\lambda_i^C(r)\bigr)
  =
  4\sinh^2(rz/2).
\]
\end{lemma}
\end{samepage}

\begin{proof}
For paths, the monic polynomial $q_r(x)=\prod_{i=1}^{r-1}(x+\lambda_i^P(r))$ satisfies the Chebyshev recurrence $q_{r+1}(x)=(x+2)q_r(x)-q_{r-1}(x)$, with $q_1(x)=1$ and $q_2(x)=x+2$.  The substitution $x=2\cosh z-2$ solves this recurrence as $q_r(x)=\sinh(rz)/\sinh z$.

For cycles, the standard identity
\[
  \prod_{i=0}^{r-1}\left(2\cosh z-2\cos\frac{2\pi i}{r}\right)
  =
  2\cosh(rz)-2
  =
  4\sinh^2(rz/2)
\]
evaluates the product including $i=0$. Removing its $i=0$ factor $x=4\sinh^2(z/2)$ leaves the positive-mode formula.
\end{proof}

\subsection{Two-dimensional product formulas}\label{sec:two-dimensional-products}

For $0\le x\le1$, set
\[
  \omega(x)=\arcosh(2-\cos2\pi x)=2\arsinh|\sin\pi x|,
  \qquad \alpha=\arsinh(1)=\log(1+\sqrt2).
\]
Then $\omega(x)=\omega(1-x)$, $\omega(1/2)=2\alpha$, and $\omega$ is increasing on $[0,1/2]$. The two-dimensional specializations of the Tzeng--Wu formulas~\cite{TzengWu2000} take the following form after the one-dimensional zero modes are evaluated.

\begin{proposition}\label{prop:two-dimensional-formulas}
For $r,s\ge1$,
\[
  \tau(\Path_r\square\Path_s)
  =
  \prod_{i=1}^{r-1}\prod_{j=1}^{s-1}
  \bigl(\lambda_i^P(r)+\lambda_j^P(s)\bigr).
\]
For $r\ge1$ and $s\ge3$,
\[
  \tau(\Path_r\square\Cycle_s)
  =
  s\prod_{k=1}^{s-1}
  \frac{\sinh(r \omega(k/s))}{\sinh \omega(k/s)}.
\]
For $r,s\ge3$,
\[
  \tau(\Cycle_r\square\Cycle_s)
  =
  rs
  \prod_{k=1}^{s-1}
  \left[
  \frac{\sinh\left(\frac r2 \omega(k/s)\right)}
  {\sinh\left(\frac12 \omega(k/s)\right)}
  \right]^2.
\]
\end{proposition}

\begin{proof}
We separate the zero-coordinate modes from those with both coordinates positive, since the former determine the prefactors. After division by $rs$, their contributions are
\[
  \frac{rs}{rs}=1,\qquad
  \frac{rs^2}{rs}=s,\qquad
  \frac{r^2s^2}{rs}=rs
\]
for the rectangle, cylinder, and torus, respectively, by \Cref{lem:oned-products}. For rectangles, the remaining positive-coordinate modes are already the displayed pairwise sums. For either product with a cycle factor, fix $1\le k\le s-1$ and use
\[
  \lambda_k^C(s)=2\cosh \omega(k/s)-2.
\]
The identities in \Cref{lem:path-cycle-polynomials} now evaluate the product over the other positive coordinate. Multiplication over $k$, with the prefactors just computed, gives the cylinder and torus formulas.
\end{proof}

\section{Logarithmic Majorization and Balancing}\label{sec:higher-dimensional}

At fixed volume, the logarithms of the side lengths have a fixed sum. Majorization compares how that sum is distributed: the larger coordinates of a majorizing vector carry at least as much of the total. This order gives a global balancing theorem for a block of paths or cycles, with every other graph factor held fixed.

\begin{definition}[Logarithmic majorization]
For $x=(x_1,\ldots,x_d)\in\R^d$, let $x_1^\downarrow\ge\cdots\ge x_d^\downarrow$ be its decreasing rearrangement. For positive integer tuples $\mathbf n=(n_1,\ldots,n_d)$ and $\mathbf m=(m_1,\ldots,m_d)$ with equal product, say that $\mathbf n$ \emph{log-majorizes} $\mathbf m$ if $(\log n_1,\ldots,\log n_d)$ majorizes $(\log m_1,\ldots,\log m_d)$, that is, if
\[
  \sum_{i=1}^k(\log n_i)^\downarrow
  \ge
  \sum_{i=1}^k(\log m_i)^\downarrow
  \qquad 1\le k<d,
\]
and equality holds for $k=d$.
\end{definition}

\begin{theorem}[Balancing with a fixed Cartesian factor]\label{thm:log-majorization}
Let $H$ be a finite connected graph, and let $\mathbf n,\mathbf m$ be positive integer $d$-tuples with equal product. If $\mathbf n$ log-majorizes $\mathbf m$, then
\[
  \tau(H\square\Path_{\mathbf n})
  \le \tau(H\square\Path_{\mathbf m}).
\]
If every side length is at least $3$, then also
\[
  \tau(H\square\Cycle_{\mathbf n})
  \le \tau(H\square\Cycle_{\mathbf m}).
\]
In either comparison, equality holds exactly when the tuples agree up to permutation.
\end{theorem}

The one-vertex choice of $H$ gives the pure-product comparison; a general fixed factor $H$ may itself contain paths and cycles, so the same theorem balances either boundary type within a mixed product. At perfect-power block volume, it identifies the equal-sided maximizer.

\begin{corollary}[Perfect-power volume]\label{thm:perfect-power-heat}
Let $H$ be a finite connected graph and let $n_1\cdots n_d=n^d$, with positive integer side lengths. Then
\[
  \tau(H\square\Path_{n_1}\square\cdots\square\Path_{n_d})
  \le \tau(H\square\Path_n^{\square d}),
\]
with equality exactly when all $n_i=n$. The analogous assertion holds for cycle factors when $n,n_i\ge3$.
\end{corollary}

\begin{proof}
The vector $(\log n_1,\ldots,\log n_d)$ has mean $\log n$ and majorizes the constant vector with that mean. The assertion and its equality case follow from \Cref{thm:log-majorization}.
\end{proof}

\subsection{Heat traces and determinants}

Kirchhoff's theorem expresses the spanning-tree count through a product of eigenvalues. We use heat traces because they are sums of exponentials and multiply under Cartesian products of graphs. Taking logarithms reduces the product comparison to convexity for a single path or cycle in logarithmic side length. Once this is available, Karamata's inequality orders the traces, and integrating their difference gives the tree-count comparison. The interpolation is the delicate step. Graph lengths are integers, while convexity is a property on a real interval. For cycles, the Fourier--Bessel formula and the Hartman--Watson representation give an exact positive mixture of Gaussian theta functions; their known logarithmic convexity~\cite[Theorem~2.3]{FaulhuberSteinerberger2017} survives the mixing. The path interpolation follows with an additional positive term.

For a connected graph $G$ on $N$ vertices, let $\lambda_0=0,\lambda_1,\ldots,\lambda_{N-1}$ be its Laplacian eigenvalues, with $\lambda_j>0$ for $j\ge1$. Write $Q(G)=\prod_{j=1}^{N-1}\lambda_j$, so that $\tau(G)=Q(G)/N$. Its heat trace is $\Tr(e^{-tL_G})=\sum_{j=0}^{N-1}e^{-t\lambda_j}$. The identity below fixes the direction of comparison: at equal vertex counts, a larger heat trace for every $t>0$ gives a smaller spanning-tree count.

\begin{lemma}\label{lem:heat-trace-determinant}
If $G$ and $H$ are connected graphs with the same number of vertices, then
\[
  \log Q(H)-\log Q(G)
  =
  \int_0^\infty
  \left[
  \Tr(e^{-tL_G})-\Tr(e^{-tL_H})
  \right]\frac{\dd t}{t}.
\]
\end{lemma}

\begin{proof}
Apply $\log a-\log b=\int_0^\infty (e^{-bt}-e^{-at})\,\dd t/t$ to the positive Laplacian eigenvalues of $H$ and $G$ and sum. Connectedness cancels the two zero modes. At $t=0$, equal vertex counts cancel the constant trace terms, leaving an $O(t)$ difference and a convergent integral; at infinity, all positive modes decay exponentially.
\end{proof}

For $X=P$ or $X=C$, let $h_X(m,t)=\Tr(e^{-tL(X_m)})$. Since Laplacian eigenvalues add under Cartesian products,
\[
  \Tr(e^{-tL(X_{\mathbf n})})=\prod_{i=1}^d h_X(n_i,t).
\]
We interpolate $h_X(m,t)$ to real $m$ so that its logarithm is strictly convex in $\log m$. This extension assigns no graph or tree count to a noninteger length. A fixed Cartesian factor preserves the resulting comparison by multiplying both heat traces by the same positive function.

\subsection{Logarithmic convexity of the one-dimensional traces}

We construct the cycle interpolation first, prove that the positive theta mixture preserves strictness, and then recover the path trace from its relation to the cycle trace.

\begin{lemma}\label{lem:fractional-cycle-heat}
For $t>0$ and $x>0$, define
\[
  \mathcal H_C(x,t)
  =
  x e^{-2t}\sum_{q\in\Z}I_{\lvert q\rvert x}(2t),
\]
where $I_\nu$ is the modified Bessel function of the first kind.  If $m\ge3$ is an integer, then $\mathcal H_C(m,t)=h_C(m,t)$.  Moreover
\[
  u\longmapsto \log \mathcal H_C(e^u,t)
\]
is strictly convex on $\R$.
\end{lemma}

\begin{proof}
\emph{Agreement at integer side lengths.} The Fourier--Bessel expansion
\[
  e^{r\cos\theta}=\sum_{\ell\in\Z}I_\ell(r)e^{i\ell\theta}
\]
and the roots-of-unity filter give, for integer $m\ge3$,
\[
  h_C(m,t)
  =
  e^{-2t}\sum_{k=0}^{m-1}e^{2t\cos(2\pi k/m)}
  =
  m e^{-2t}\sum_{q\in\Z}I_{\lvert q\rvert m}(2t).
\]

\emph{A positive theta representation.} For real side lengths, the length dependence enters through the Bessel order. We express that dependence as a Gaussian mixture using the Hartman--Watson representation~\cite{HartmanWatson1974,Yor1980}: for each $R>0$ there is a probability measure $\mu_R$ supported on $(0,\infty)$ such that, for $\nu\ge0$,
\[
  \frac{I_\nu(R)}{I_0(R)}
  =
  \int_0^\infty e^{-\nu^2\sigma/2}\,\mu_R(\dd\sigma).
\]
The Bessel series is absolutely convergent, and its represented terms are nonnegative. Tonelli's theorem therefore gives
\[
  \mathcal H_C(x,t)
  =
  e^{-2t}I_0(2t)
  \int_0^\infty
  x\sum_{q\in\Z}e^{-\sigma q^2x^2/2}\,\mu_{2t}(\dd\sigma).
\]
The Gaussian factors require a rescaled form of \cite[Theorem~2.3]{FaulhuberSteinerberger2017}. We give a direct variance proof, with Poisson summation for the remaining parameter range, to make pointwise strictness explicit. After mixing, this strictness determines the equality case in the majorization theorem.

\emph{Convexity of the theta factors.} For $a>0$ put
\[
  \Theta_a(x)=x\sum_{q\in\Z}e^{-a q^2x^2}.
\]
Let
\[
  \theta(c)=\sum_{q\in\Z}e^{-cq^2},
  \qquad
  \psi(c)=-c\frac{\theta'(c)}{\theta(c)}.
\]
Since $\Theta_a(e^u)=e^u\theta(ae^{2u})$, differentiation gives
\[
  \frac{\dd^2}{\dd u^2}\log\Theta_a(e^u)=-4c\psi'(c),
  \qquad c=ae^{2u}.
\]
The target is now $\psi'(c)<0$ for all $c>0$. To establish this sign, let $K$ have the discrete Gaussian distribution
\[
  \mathbb P_c(K=k)=\frac{e^{-ck^2}}{\theta(c)}.
\]
With $Y=K^2$, we have $\psi(c)=c\,\mathbb E_c[Y]$ and
\[
  \psi'(c)=\mathbb E_c[Y]-c\,\operatorname{Var}_c(Y).
\]
For $c\ge\pi$, the inequality $Y^2\ge Y$ gives
\[
  \operatorname{Var}_c(Y)
  \ge
  \mathbb E_c[Y](1-\mathbb E_c[Y]).
\]
To bound the mean, use $\theta(c)\ge1$ and $k^2\ge k$ for $k\ge1$:
\[
  \mathbb E_c[Y]
  \le
  2\sum_{k\ge1}k^2e^{-\pi k}
  =
  \frac{2e^{-\pi}(1+e^{-\pi})}{(1-e^{-\pi})^3}
  <1-\frac1\pi
  \le
  1-\frac1c.
\]
Consequently, $\psi'(c)\le\mathbb E_c[Y]\bigl(1-c(1-\mathbb E_c[Y])\bigr)<0$.

The interval $0<c<\pi$ follows by duality. Poisson summation and its logarithmic derivative give
\[
  \theta(c)=\sqrt{\frac{\pi}{c}}\,\theta\left(\frac{\pi^2}{c}\right),
  \qquad
  \psi(c)+\psi\left(\frac{\pi^2}{c}\right)=\frac12.
\]
Differentiating the second identity yields
\[
  \psi'(c)=\frac{\pi^2}{c^2}\psi'\left(\frac{\pi^2}{c}\right)<0,
\]
since $\pi^2/c>\pi$. Thus $\psi'(c)<0$ for all $c>0$, proving that $\log\Theta_a(e^u)$ is strictly convex.

\emph{Passing to the mixture.} The theta inequality is strict at every positive mixing parameter. This is the point that H\"older's inequality alone would not supply. For $u\ne v$, $0<\lambda<1$, and $w=(1-\lambda)u+\lambda v$, strict log-convexity of $\Theta_{\sigma/2}$ gives, for every $\sigma>0$,
\[
  \Theta_{\sigma/2}(e^w)
  <
  \Theta_{\sigma/2}(e^u)^{1-\lambda}
  \Theta_{\sigma/2}(e^v)^\lambda.
\]
The probability measure $\mu_{2t}$ is supported on $(0,\infty)$, so this strict inequality holds on a set of measure one. Integration followed by H\"older's inequality yields
\[
  \int \Theta_{\sigma/2}(e^w)\,\mu_{2t}(\dd\sigma)
  <
  \left(\int \Theta_{\sigma/2}(e^u)\,\mu_{2t}(\dd\sigma)\right)^{1-\lambda}
  \left(\int \Theta_{\sigma/2}(e^v)\,\mu_{2t}(\dd\sigma)\right)^\lambda.
\]
Multiplication by the positive factor $e^{-2t}I_0(2t)$ preserves strictness, completing the convexity proof for $\log\mathcal H_C(e^u,t)$.
\end{proof}

Only the two endpoint angles separate the path trace from the paired cycle Fourier sum.

\begin{lemma}\label{lem:fractional-path-heat}
For $t>0$ and $x>0$, define
\[
  \mathcal H_P(x,t)=\frac12\mathcal H_C(2x,t)+\frac12(1-e^{-4t}).
\]
For every positive integer $m$, this equals $h_P(m,t)$ at $x=m$, and $u\mapsto\log\mathcal H_P(e^u,t)$ is strictly convex on $\R$.
\end{lemma}

\begin{proof}
The Fourier--Bessel expansion and roots-of-unity filter apply to $2m$ equally spaced angles. Pairing $\pi k/m$ with $2\pi-\pi k/m$ gives, for every integer $m\ge1$,
\[
  \mathcal H_C(2m,t)
  =e^{-2t}\sum_{k=0}^{2m-1}e^{2t\cos(\pi k/m)}
  =1+e^{-4t}+2\sum_{k=1}^{m-1}e^{-t(2-2\cos(\pi k/m))}.
\]
The path heat trace is one plus the same sum, so the formula agrees with $h_P(m,t)$. The endpoint $m=1$ is included: the sum is empty, giving $\mathcal H_P(1,t)=1$.

By \Cref{lem:fractional-cycle-heat}, $G(u)=\mathcal H_C(2e^u,t)/2$ is strictly log-convex. Set $c_t=(1-e^{-4t})/2>0$. For distinct $u,v$, $0<\lambda<1$, and $w=(1-\lambda)u+\lambda v$,
\[
  c_t+G(w)
  <c_t+G(u)^{1-\lambda}G(v)^\lambda
  \le(c_t+G(u))^{1-\lambda}(c_t+G(v))^\lambda,
\]
where the second inequality is the two-term H\"older inequality. This proves strict log-convexity of the path interpolation.
\end{proof}

\begin{proof}[Proof of \Cref{thm:log-majorization}]
Fix $t>0$ and take $X=P$ or $X=C$, with admissible side lengths. By \Cref{lem:fractional-cycle-heat,lem:fractional-path-heat}, $f_t(u)=\log\mathcal H_X(e^u,t)$ is strictly convex, so Karamata's inequality gives
\[
  \sum_i\log h_X(n_i,t)\ge\sum_i\log h_X(m_i,t),
\]
strictly unless the tuples are permutations. Exponentiating restores the product traces; multiplication by $h_H(t)=\Tr(e^{-tL_H})$, which is positive, adjoins the fixed factor. The resulting graphs are connected and have the same number of vertices, so \Cref{lem:heat-trace-determinant} gives
\[
\begin{split}
 &\log Q(H\square X_{\mathbf m})-\log Q(H\square X_{\mathbf n})\\
 &\qquad=\int_0^\infty h_H(t)
   \left[\prod_i h_X(n_i,t)-\prod_i h_X(m_i,t)\right]\frac{\dd t}{t}.
\end{split}
\]
For nonpermutation tuples the integrand is positive for every $t>0$, and division by the common vertex count gives the strict tree-count comparison. Permuting factors gives isomorphic graphs, accounting for equality.
\end{proof}

\subsection{Spectral consequences of the heat-trace order}\label{sec:spectral-consequences}

The trace comparison has consequences beyond tree counts. Integrating against a nonnegative measure gives spectral sums built from mixtures of exponentials, which are exactly the completely monotone functions: smooth functions $f:(0,\infty)\to\R$ with $(-1)^k f^{(k)}\ge0$ for every integer $k\ge0$. Bernstein's theorem supplies the representation used below~\cite[Theorem~1.4]{SchillingSongVondracek2012}.

\begin{samepage}
\begin{corollary}\label{cor:spectral-functionals}
Let $H$ be a finite connected graph, and let $\mathbf n,\mathbf m$ be positive integer $d$-tuples with equal product such that $\mathbf n$ log-majorizes $\mathbf m$. Choose $X=P$, or $X=C$ with all side lengths at least $3$, and put $G_{\mathbf n}=H\square X_{\mathbf n}$ and $G_{\mathbf m}=H\square X_{\mathbf m}$. Then
\begin{equation}\label{eq:heat-trace-order}
  \Tr(e^{-tL_{G_{\mathbf n}}})\ge\Tr(e^{-tL_{G_{\mathbf m}}}),
  \qquad t>0,
\end{equation}
strictly unless the tuples are permutations. If
\[
  f(\lambda)=\int_{[0,\infty)}e^{-t\lambda}\,\nu(\dd t)<\infty,
  \qquad \lambda>0,
\]
for a nonnegative Borel measure $\nu$, then
\[
  \sum_{\lambda\in\Spec^+(G_{\mathbf n})}f(\lambda)
  \ge
  \sum_{\lambda\in\Spec^+(G_{\mathbf m})}f(\lambda).
\]
The same comparison holds over the full spectra if $f(0):=\nu([0,\infty))<\infty$. If $\nu((0,\infty))>0$, equality in either spectral-sum comparison holds exactly for permutation tuples.
\end{corollary}
\end{samepage}

\begin{proof}
The Karamata argument in the proof of \Cref{thm:log-majorization} gives \eqref{eq:heat-trace-order} and its strictness. Both graphs are connected, so Tonelli's theorem gives
\[
  \sum_{\lambda\in\Spec^+(G)}f(\lambda)
  =\int_{[0,\infty)}\bigl(\Tr(e^{-tL_G})-1\bigr)\,\nu(\dd t)
\]
for either graph $G$. Finiteness of the positive spectra makes both integrals finite, so the heat-trace order can be integrated directly. Positive measure on $(0,\infty)$ preserves strictness; adding the common finite contribution $f(0)$ gives the full-spectrum assertion. Permutation tuples give isomorphic graphs and hence equality.
\end{proof}

For example, the identity $\lambda^{-p}=\Gamma(p)^{-1}\int_0^\infty t^{p-1}e^{-t\lambda}\dd t$ gives, for the same graphs and every $p>0$,
\[
  \sum_{\lambda\in\Spec^+(G_{\mathbf n})}\lambda^{-p}
  \ge
  \sum_{\lambda\in\Spec^+(G_{\mathbf m})}\lambda^{-p}.
\]

In particular, taking $p=1$ compares the Kirchhoff index $\operatorname{Kf}(G)=|V(G)|\sum_{\lambda\in\Spec^+(G)}\lambda^{-1}$, the sum of effective resistances over unordered vertex pairs when each edge has unit resistance~\cite{GutmanMohar1996}. Since $G_{\mathbf n}$ and $G_{\mathbf m}$ have the same vertex count, balancing decreases this index, strictly unless the side-length tuples are permutations. The equal-sided block therefore uniquely minimizes it at perfect-power block volume within the corresponding product family, with $H$ fixed.

Applying the corollary to $f(\lambda)=(\mu+\lambda)^{-p}$ also gives
\[
  \Tr(\mu\mathrm{Id}+L_{G_{\mathbf n}})^{-p}
  \ge\Tr(\mu\mathrm{Id}+L_{G_{\mathbf m}})^{-p},
  \qquad \mu,p>0.
\]
Both inequalities are strict for nonpermutation tuples. These functionals are minimized together by the equal-sided block at perfect-power block volume, uniquely for nonconstant $f$. At other volumes, the result compares only tuples related by logarithmic majorization.

A positive spectral shift multiplies the heat trace by $e^{-\mu t}$ and also preserves the determinant comparison.

\begin{proposition}\label{prop:shifted-majorization}
Let $H$ be a finite connected graph, and let $\mathbf n,\mathbf m$ be positive integer $d$-tuples with equal product such that $\mathbf n$ log-majorizes $\mathbf m$. Choose $X=P$, or choose $X=C$ with every side length at least $3$. For every $\mu>0$,
\[
  \det(\mu\mathrm{Id}+L_{H\square X_{\mathbf n}})
  \le \det(\mu\mathrm{Id}+L_{H\square X_{\mathbf m}}),
\]
with equality exactly for permutation tuples. At $\mu=0$, the same assertion holds for the products of positive eigenvalues.
\end{proposition}

\begin{proof}
For $\mu>0$, Frullani's formula applied to all eigenvalues gives
\[
\begin{split}
 &\log\det(\mu\mathrm{Id}+L_{H\square X_{\mathbf m}})
   -\log\det(\mu\mathrm{Id}+L_{H\square X_{\mathbf n}})\\
 &\qquad=\int_0^\infty e^{-\mu t}\Tr(e^{-tL_H})
   \left[\prod_i h_X(n_i,t)-\prod_i h_X(m_i,t)\right]\frac{\dd t}{t}.
\end{split}
\]
Equal vertex counts give convergence at zero, and the positive shift gives convergence at infinity. The heat-trace order in \Cref{cor:spectral-functionals} makes the integral positive unless the tuples are permutations, proving the positive-shift assertion and its equality case. At zero shift, \Cref{thm:log-majorization} and the common vertex count give the comparison of positive-eigenvalue products.
\end{proof}

\subsection{Pairwise balancing and arithmetic obstructions}

For two factors, the shifted comparison takes a simple form. Define $D_\mu(r,s)=\det(\mu\mathrm{Id}+L_{P_r\square P_s})$ and $E_\mu(r,s)=\det(\mu\mathrm{Id}+L_{C_r\square C_s})$ for $\mu>0$, and use the positive-eigenvalue products at $\mu=0$.

\begin{corollary}\label{thm:D-balancing}\label{thm:E-balancing}
Let $A,B,a,b$ be positive integers satisfying $AB=ab$ and $A\le a\le b\le B$. For every $\mu\ge0$,
\[
  D_\mu(a,b)\ge D_\mu(A,B).
\]
If $A\ge3$, then also $E_\mu(a,b)\ge E_\mu(A,B)$. Both inequalities are strict for distinct pairs.
\end{corollary}

\begin{proof}
The pair $(\log A,\log B)$ majorizes $(\log a,\log b)$. Apply \Cref{prop:shifted-majorization} with the one-vertex graph as $H$.
\end{proof}

\begin{corollary}\label{thm:path-higher}\label{thm:cycle-higher}
Let $H$ be a finite connected graph, and let $A\le a\le b\le B$ be positive integers with $AB=ab$ and $(A,B)\ne(a,b)$. Then
\[
  \tau(P_a\square P_b\square H)>\tau(P_A\square P_B\square H).
\]
The same inequality holds with $P$ replaced by $C$ when $A\ge3$.
\end{corollary}

\begin{proof}
Apply \Cref{thm:log-majorization} to the two changed coordinates.
\end{proof}

With integer side lengths, the available divisors limit each move. Repeated pairwise balancing can therefore stop before the coordinates become equal.

\begin{definition}
A divisor tuple $(n_1,\ldots,n_d)$ of fixed product is \emph{pairwise balanced} if no pair of coordinates can be replaced by a strictly more balanced divisor pair with the same product; that is, for no $i<j$ with $\min(n_i,n_j)=A$ and $\max(n_i,n_j)=B$ is there a factorization $AB=ab$ with $A<a\le b<B$.
\end{definition}

\begin{corollary}\label{cor:pairwise-maximizers}
For a fixed number of factors, every fixed-volume maximizer in the pure path class is pairwise balanced.  Every fixed-volume maximizer in the pure cycle class, with all side lengths at least $3$, is pairwise balanced.  If a fixed product has a unique terminal pairwise-balanced tuple up to permutation, then that tuple is the unique maximizer up to permutation.
\end{corollary}

\begin{proof}
A nontrivial balancing of two coordinates would strictly increase $\tau$ by \Cref{thm:path-higher} without changing the volume, contradicting maximality. We prove termination by tracking deviation from the fixed logarithmic mean: let $\bar u=d^{-1}\sum_{i=1}^d\log n_i=d^{-1}\log(n_1\cdots n_d)$, fixed under product-preserving moves; each strict move decreases
\[
  \sum_{i=1}^d(\log n_i-\bar u)^2
\]
because the two changed logarithmic side lengths $x\le y$ have fixed sum and satisfy
\[
  (x-\bar u)^2+(y-\bar u)^2
  =
  2\left(\frac{x+y}{2}-\bar u\right)^2+\frac{(y-x)^2}{2}.
\]
Bringing the pair closer strictly decreases the last term. Because a fixed product has only finitely many divisor tuples, the moves must terminate at a pairwise-balanced tuple. If the terminal tuple is unique up to permutation, every other tuple reaches it along moves that strictly increase $\tau$, giving the unique maximizer up to permutation.
\end{proof}

\paragraph{Limits of pairwise balancing.}

\begin{proposition}\label{prop:perfect-power-obstruction}
The tuple $(50,54,64,75)$ has product $60^4$ and is pairwise balanced.  Consequently, perfect-power volume alone does not imply that $(n,\ldots,n)$ is the unique terminal tuple under pairwise balancing.
\end{proposition}

\begin{proof}
The product is $50\cdot54\cdot64\cdot75=12960000=60^4$. For a pair $x<y$, a nontrivial balancing move exists exactly when $xy$ has a divisor $a$ with $x<a\le\sqrt{xy}$; the replacement is then $(a,xy/a)$. \Cref{fig:pairwise-obstruction} displays the six candidate intervals. Every integer candidate fails the corresponding divisibility test. Thus no pair of coordinates can be replaced by a strictly more balanced divisor pair.
\end{proof}

\begin{figure}[tbp]
\centering
\begin{tikzpicture}[x=3.05cm,y=1.0cm,every node/.style={font=\footnotesize}]
\filldraw[fill=gray!7,draw=gray!40,line width=.25pt] (0,3) rectangle (1,4);
\filldraw[fill=white,draw=gray!40,line width=.25pt] (1,3) rectangle (2,4);
\node at (1.500,3.780) {$(50,\sqrt{2700}]$};
\draw[gray!45,line width=4pt] (1.10000,3.460)--(1.23964,3.460);
\draw[gray,line width=.35pt] (1.23964,3.380)--(1.23964,3.540);
\draw[NavyBlue,fill=white,line width=.35pt] (1.17119,3.460) circle[radius=.8pt];
\filldraw[fill=white,draw=gray!40,line width=.25pt] (2,3) rectangle (3,4);
\node at (2.500,3.780) {$(50,\sqrt{3200}]$};
\draw[gray!45,line width=4pt] (2.10000,3.460)--(2.56763,3.460);
\draw[gray,line width=.35pt] (2.56763,3.380)--(2.56763,3.540);
\draw[NavyBlue,fill=white,line width=.35pt] (2.17119,3.460) circle[radius=.8pt];
\draw[NavyBlue,fill=white,line width=.35pt] (2.24238,3.460) circle[radius=.8pt];
\draw[NavyBlue,fill=white,line width=.35pt] (2.31358,3.460) circle[radius=.8pt];
\draw[NavyBlue,fill=white,line width=.35pt] (2.38477,3.460) circle[radius=.8pt];
\draw[NavyBlue,fill=white,line width=.35pt] (2.45596,3.460) circle[radius=.8pt];
\draw[NavyBlue,fill=white,line width=.35pt] (2.52715,3.460) circle[radius=.8pt];
\filldraw[fill=white,draw=gray!40,line width=.25pt] (3,3) rectangle (4,4);
\node at (3.500,3.780) {$(50,\sqrt{3750}]$};
\draw[gray!45,line width=4pt] (3.10000,3.460)--(3.90000,3.460);
\draw[gray,line width=.35pt] (3.90000,3.380)--(3.90000,3.540);
\draw[NavyBlue,fill=white,line width=.35pt] (3.17119,3.460) circle[radius=.8pt];
\draw[NavyBlue,fill=white,line width=.35pt] (3.24238,3.460) circle[radius=.8pt];
\draw[NavyBlue,fill=white,line width=.35pt] (3.31358,3.460) circle[radius=.8pt];
\draw[NavyBlue,fill=white,line width=.35pt] (3.38477,3.460) circle[radius=.8pt];
\draw[NavyBlue,fill=white,line width=.35pt] (3.45596,3.460) circle[radius=.8pt];
\draw[NavyBlue,fill=white,line width=.35pt] (3.52715,3.460) circle[radius=.8pt];
\draw[NavyBlue,fill=white,line width=.35pt] (3.59834,3.460) circle[radius=.8pt];
\draw[NavyBlue,fill=white,line width=.35pt] (3.66953,3.460) circle[radius=.8pt];
\draw[NavyBlue,fill=white,line width=.35pt] (3.74073,3.460) circle[radius=.8pt];
\draw[NavyBlue,fill=white,line width=.35pt] (3.81192,3.460) circle[radius=.8pt];
\draw[NavyBlue,fill=white,line width=.35pt] (3.88311,3.460) circle[radius=.8pt];
\filldraw[fill=gray!7,draw=gray!40,line width=.25pt] (0,2) rectangle (1,3);
\filldraw[fill=gray!7,draw=gray!40,line width=.25pt] (1,2) rectangle (2,3);
\filldraw[fill=white,draw=gray!40,line width=.25pt] (2,2) rectangle (3,3);
\node at (2.500,2.780) {$(54,\sqrt{3456}]$};
\draw[gray!45,line width=4pt] (2.10000,2.460)--(2.44085,2.460);
\draw[gray,line width=.35pt] (2.44085,2.380)--(2.44085,2.540);
\draw[NavyBlue,fill=white,line width=.35pt] (2.17119,2.460) circle[radius=.8pt];
\draw[NavyBlue,fill=white,line width=.35pt] (2.24238,2.460) circle[radius=.8pt];
\draw[NavyBlue,fill=white,line width=.35pt] (2.31358,2.460) circle[radius=.8pt];
\draw[NavyBlue,fill=white,line width=.35pt] (2.38477,2.460) circle[radius=.8pt];
\filldraw[fill=white,draw=gray!40,line width=.25pt] (3,2) rectangle (4,3);
\node at (3.500,2.780) {$(54,\sqrt{4050}]$};
\draw[gray!45,line width=4pt] (3.10000,2.460)--(3.78626,2.460);
\draw[gray,line width=.35pt] (3.78626,2.380)--(3.78626,2.540);
\draw[NavyBlue,fill=white,line width=.35pt] (3.17119,2.460) circle[radius=.8pt];
\draw[NavyBlue,fill=white,line width=.35pt] (3.24238,2.460) circle[radius=.8pt];
\draw[NavyBlue,fill=white,line width=.35pt] (3.31358,2.460) circle[radius=.8pt];
\draw[NavyBlue,fill=white,line width=.35pt] (3.38477,2.460) circle[radius=.8pt];
\draw[NavyBlue,fill=white,line width=.35pt] (3.45596,2.460) circle[radius=.8pt];
\draw[NavyBlue,fill=white,line width=.35pt] (3.52715,2.460) circle[radius=.8pt];
\draw[NavyBlue,fill=white,line width=.35pt] (3.59834,2.460) circle[radius=.8pt];
\draw[NavyBlue,fill=white,line width=.35pt] (3.66953,2.460) circle[radius=.8pt];
\draw[NavyBlue,fill=white,line width=.35pt] (3.74073,2.460) circle[radius=.8pt];
\filldraw[fill=gray!7,draw=gray!40,line width=.25pt] (0,1) rectangle (1,2);
\filldraw[fill=gray!7,draw=gray!40,line width=.25pt] (1,1) rectangle (2,2);
\filldraw[fill=gray!7,draw=gray!40,line width=.25pt] (2,1) rectangle (3,2);
\filldraw[fill=white,draw=gray!40,line width=.25pt] (3,1) rectangle (4,2);
\node at (3.500,1.780) {$(64,\sqrt{4800}]$};
\draw[gray!45,line width=4pt] (3.10000,1.460)--(3.47604,1.460);
\draw[gray,line width=.35pt] (3.47604,1.380)--(3.47604,1.540);
\draw[NavyBlue,fill=white,line width=.35pt] (3.17119,1.460) circle[radius=.8pt];
\draw[NavyBlue,fill=white,line width=.35pt] (3.24238,1.460) circle[radius=.8pt];
\draw[NavyBlue,fill=white,line width=.35pt] (3.31358,1.460) circle[radius=.8pt];
\draw[NavyBlue,fill=white,line width=.35pt] (3.38477,1.460) circle[radius=.8pt];
\draw[NavyBlue,fill=white,line width=.35pt] (3.45596,1.460) circle[radius=.8pt];
\filldraw[fill=gray!7,draw=gray!40,line width=.25pt] (0,0) rectangle (1,1);
\filldraw[fill=gray!7,draw=gray!40,line width=.25pt] (1,0) rectangle (2,1);
\filldraw[fill=gray!7,draw=gray!40,line width=.25pt] (2,0) rectangle (3,1);
\filldraw[fill=gray!7,draw=gray!40,line width=.25pt] (3,0) rectangle (4,1);
\node[font=\small] at (0.50,4.19) {50};
\node[font=\small] at (1.50,4.19) {54};
\node[font=\small] at (2.50,4.19) {64};
\node[font=\small] at (3.50,4.19) {75};
\node[font=\small,anchor=east] at (-.08,3.50) {50};
\node[font=\small,anchor=east] at (-.08,2.50) {54};
\node[font=\small,anchor=east] at (-.08,1.50) {64};
\node[font=\small,anchor=east] at (-.08,0.50) {75};
\node[align=left,anchor=north west,font=\small,text=gray!80!black] at (.18,1.34) {6 pairs\\No admissible divisor};
\node[font=\footnotesize] at (1.88,-.27) {Dots: integer candidates. None divides $xy$.};
\end{tikzpicture}
\caption{No pairwise balancing move is available for $(50,54,64,75)$. In row $x$ and column $y$, the bar shows $(x,\sqrt{xy}]$ and the dots mark integers. No marked integer divides $xy$. Bar lengths share the same scale.}
\label{fig:pairwise-obstruction}
\end{figure}

\begin{remark}
The obstruction in \Cref{prop:perfect-power-obstruction} is arithmetic: every pair admits a real balancing move, but none admits an intermediate integer divisor. Logarithmic majorization compares the full tuples directly. Every four-tuple of product $60^4$ has a logarithmic vector that majorizes $(\log60,\ldots,\log60)$, so \Cref{thm:perfect-power-heat} still makes the uniform tuple the unique maximizer. Pairwise moves alone do not prove this conclusion.
\end{remark}

\section{Quantitative Comparisons for Rectangles and Tori}\label{sec:rectangles-tori}

The heat-trace comparison in \Cref{sec:higher-dimensional} determines the sign of a balancing gain. To measure its size, we turn to exact formulas in two dimensions. At general volumes the order remains partial, and in mixed products it leaves the allocation between boundary types unresolved. The formulas separate these issues: the torus decomposition measures gains within one family, while the cylinder corrections locate the relaxed optimum and show how available divisors constrain the graph optimum.

\subsection{Rectangles with Two Free Factors}

The known rectangle comparison supplies the quantitative benchmark: its gain has a sharp boundary contribution.

\begin{samepage}
\begin{theorem}[{\cite[Theorem~1.2]{Zhang2026Balancing}}]\label{thm:rectangle-balancing}
Let $A,B,a,b$ be positive integers satisfying
\[
  AB=ab,\qquad A\le a\le b\le B.
\]
Then
\[
  \log\tau(\Path_a\square\Path_b)-\log\tau(\Path_A\square\Path_B)
  \ge
  \arsinh(1)\bigl((A+B)-(a+b)\bigr).
\]
If the rectangles are distinct, the inequality is strict.  The universal constant $\arsinh(1)$ is optimal.
\end{theorem}
\end{samepage}

Both \Cref{thm:rectangle-balancing} and \Cref{thm:log-majorization} with two path factors and the one-vertex graph as $H$ show that, among divisor rectangles $\Path_r\square\Path_s$ with $rs=N$, the closest-to-square divisor pair uniquely maximizes $\tau$, up to rotation.

\subsection{The periodic correction}\label{sec:periodic-trapezoids}

The periodic contribution is a sampling error. To measure it, write
\[
  \gamma(\theta)=\arcosh(2-\cos\theta),\qquad \theta\in\R.
\]
The function is even and $2\pi$-periodic, and agrees with $\omega(\theta/2\pi)$ on $[0,\pi]$. Its average and sampled sum are
\[
  I=\int_0^1\omega(x)\dd x=\frac1\pi\int_0^\pi\gamma(\theta)\dd\theta,
  \qquad
  S_s=\sum_{k=1}^{s-1}\omega(k/s),\quad s\in\Z_{\ge1}.
\]
The bulk constant is the classical square-lattice spanning-tree entropy~\cite{Wu1977,ShrockWu2000}:
\[
  I=\frac{4G_{\mathrm{Cat}}}{\pi},
  \qquad G_{\mathrm{Cat}}=\sum_{j=0}^\infty\frac{(-1)^j}{(2j+1)^2}.
\]
Here $G_{\mathrm{Cat}}$ is Catalan's constant.

For a direct evaluation, set $J(a)=\int_0^{\pi/2}\arsinh(a\sin\theta)\dd\theta$. Differentiating and substituting $u=\cos\theta$ yields $J'(a)=\arctan(a)/a$ for $a>0$. Starting from $J(0)=0$, integration gives $J(1)=\int_0^1\arctan(a)\dd a/a=G_{\mathrm{Cat}}$; symmetry then gives $I=4J(1)/\pi$.

\begin{lemma}\label{lem:periodic-trap}
For every integer $s\ge1$,
\[
  \frac{S_s}{s}
  =
  I+
  \frac{2}{\pi s}
  \int_0^\pi
  \log(1-e^{-s\gamma(\phi)})\dd\phi.
\]
Consequently $S_s/s$ is increasing in $s$, and
\[
  S_s=sI-\frac{\pi}{3s}+O(s^{-3}).
\]
\end{lemma}

\begin{proof}
For $a\ge1$,
\[
  \frac1\pi\int_0^\pi\log(a-\cos\phi)\dd\phi
  =
  \log\frac{a+\sqrt{a^2-1}}{2}.
\]
Taking $a=2-\cos\theta$ expresses the dispersion itself as an average:
\[
  \gamma(\theta)=\log2+
  \frac1\pi\int_0^\pi\log(2-\cos\theta-\cos\phi)\dd\phi.
\]
For fixed $\phi>0$, set $q=e^{-\gamma(\phi)}$ and factor the integrand as
\[
  2-\cos\theta-\cos\phi
  =
  \frac{1-2q\cos\theta+q^2}{2q}.
\]
The Fourier expansion on the periodic grid $\theta=2\pi k/s$ makes the sampling error explicit: averaging
\[
  \log(1-2q\cos\theta+q^2)
  =
  -2\sum_{m\ge1}\frac{q^m}{m}\cos(m\theta)
\]
retains only multiples of $s$, so
\[
  \frac1s\sum_{k=0}^{s-1}\gamma(2\pi k/s)
  =
  I+\frac{2}{\pi s}\int_0^\pi\log(1-e^{-s\gamma(\phi)})\dd\phi.
\]
Dropping the zero summand gives the stated identity.

For fixed $z>0$, the integrand $p^{-1}\log(1-e^{-pz})$ increases in $p>0$. The identity therefore gives monotonicity of $S_s/s$ directly.

Only the endpoint $\phi=0$ contributes at polynomial order. Choose $\delta>0$ small: the contribution on $[\delta,\pi]$ is $O(e^{-cs})$, since $\gamma$ is bounded below there. On $[0,\delta]$, use $\gamma(\phi)=\phi+O(\phi^3)$ and $\gamma(\phi)\asymp\phi$; with $u=s\phi$, this gives
\[
  s\gamma(u/s)=u+O(u^3/s^2),
  \qquad 0\le u\le \delta s.
\]
Moreover
\[
  \left|
  \log(1-e^{-s\gamma(u/s)})-\log(1-e^{-u})
  \right|
  \le
  \frac{C u^3}{s^2(e^{cu}-1)},
\]
which is integrable in $u$ on $[0,\infty)$. The change of variables $\dd\phi=\dd u/s$ turns this majorant into an $O(s^{-3})$ error. Extending $[0,\delta s]$ to $[0,\infty)$ adds only an exponentially small tail, leaving
\[
  \int_0^\pi\log(1-e^{-s\gamma(\phi)})\dd\phi
  =
  \frac1s\int_0^\infty\log(1-e^{-u})\dd u+O(s^{-3}).
\]
The integral on the right equals $-\pi^2/6$, proving the expansion with an absolute error constant.
\end{proof}

\subsection{Tori with Two Periodic Factors}

The qualitative inequality is already a consequence of \Cref{thm:log-majorization}. We give a second proof to expose the gain itself: the sampled dispersion contributes one term, and a positive residual contributes another. The residual comparison uses the elementary kernel below.

\begin{lemma}\label{lem:G}
Let $t>1$ and
\[
  \mathcal U_t(y)=\log\frac{1-e^{-2ty}}{1-e^{-2y}},
  \qquad y>0.
\]
Then $\mathcal U_t(y)>0$ and $\mathcal U_t$ is strictly decreasing.
\end{lemma}

\begin{proof}
Positivity follows from $t>1$. Differentiating gives
\[
  \mathcal U_t'(y)=\frac1y\left(\frac{2ty}{e^{2ty}-1}-\frac{2y}{e^{2y}-1}\right).
\]
Since $u/(e^u-1)$ is strictly decreasing on $(0,\infty)$ and $2ty>2y$, this derivative is negative.
\end{proof}

\begin{samepage}
\begin{theorem}\label{thm:torus-exact-gain}
Let $A,B,a,b$ be integers satisfying
\[
  AB=ab,\qquad 3\le A\le a\le b\le B.
\]
Set $t=a/A=B/b$.  If $t=1$, the two tori are identical.  If $t>1$, then
\[
  \log\frac{\tau(\Cycle_a\square\Cycle_b)}
  {\tau(\Cycle_A\square\Cycle_B)}
  =
  (t-1)(AS_b-bS_A)+2\Gamma_T,
\]
where
\[
  \Gamma_T=
  \sum_{k=1}^{b-1}\mathcal U_t\left(\frac A2\omega(k/b)\right)
  -
  \sum_{j=1}^{A-1}\mathcal U_t\left(\frac b2\omega(j/A)\right).
\]
\end{theorem}
\end{samepage}

\begin{proof}
For $t>1$, write
\[
  \mathcal V_t(y)=\log\frac{\sinh(ty)}{\sinh y}
  =
  (t-1)y+\mathcal U_t(y).
\]
We compare both tori through the same intermediate torus, $\Cycle_A\square\Cycle_b$, so that both changes of side length use the same ratio. The product formula gives
\[
  \log\frac{\tau(\Cycle_a\square\Cycle_b)}
  {\tau(\Cycle_A\square\Cycle_B)}
  =
  2\sum_{k=1}^{b-1}\mathcal V_t\left(\frac A2\omega(k/b)\right)
  -
  2\sum_{j=1}^{A-1}\mathcal V_t\left(\frac b2\omega(j/A)\right).
\]
The common ratio $a/A=B/b=t$ cancels the vertex prefactors. Separating the linear part from the residual gives the identity; when $t=1$, the pairs are identical.
\end{proof}

\begin{lemma}\label{lem:torus-residual-matching}
Let integers $A,B,a,b$ satisfy $AB=ab$ and $3\le A<a\le b\le B$. With $t=a/A=B/b>1$ and $\Gamma_T$ defined in \Cref{thm:torus-exact-gain}, one has $\Gamma_T>0$.
\end{lemma}

\begin{proof}
We match the terms injectively so that positive unmatched terms make the comparison strict. Since $t=a/A>1$ and $a\le b$, we have $b>A$, hence $b/A>1$; the sequence $j\mapsto\lfloor bj/A\rfloor$ increases by at least one at each integer step. For $1\le j\le A-1$, define
\[
  k_j=
  \begin{cases}
  \lfloor bj/A\rfloor,& j\le A/2,\\
  b-\lfloor b(A-j)/A\rfloor,& j>A/2.
  \end{cases}
\]
Each half is strictly increasing. The first lies at or below $b/2$ and the second above $b/2$, keeping the halves separated. Thus all $k_j$ are distinct elements of $\{1,\ldots,b-1\}$. By construction, $\dist(k_j/b,\Z)\le\dist(j/A,\Z)$. The symmetry $\omega(x)=\omega(1-x)$ and monotonicity on $[0,1/2]$ give $\omega(k_j/b)\le \omega(j/A)$, hence $(A/2)\omega(k_j/b)\le(b/2)\omega(j/A)$. Because $\mathcal U_t$ is decreasing, each matched term in the first sum dominates its partner. The remaining $b-A$ terms are positive by \Cref{lem:G}; they make $\Gamma_T>0$.
\end{proof}

\begin{remark}\label{thm:torus-balancing}
The exact gain gives a second proof of the two-dimensional torus case of \Cref{thm:log-majorization}. For integers $A,B,a,b$ with $AB=ab$ and $3\le A\le a\le b\le B$, set $t=a/A=B/b$. If $t=1$, the pairs are identical. Otherwise, \Cref{thm:torus-exact-gain} writes the gain as a sum of two nonnegative terms: \Cref{lem:periodic-trap} gives $AS_b-bS_A\ge0$, and \Cref{lem:torus-residual-matching} gives $\Gamma_T>0$. Hence $\tau(\Cycle_a\square\Cycle_b)>\tau(\Cycle_A\square\Cycle_B)$ for distinct pairs.
\end{remark}

For admissible divisors $d_1<d_2\le\sqrt N$, the substitution $A=d_1$, $B=N/d_1$, $a=d_2$, $b=N/d_2$ satisfies $3\le A<a\le b<B$, so \Cref{thm:torus-balancing} makes the closest-to-square torus $\Cycle_r\square\Cycle_s$ with $rs=N$ and $r,s\ge3$ the unique maximizer of $\tau$, up to rotation.

\subsection{Torus asymptotics and the boundary constant}

The exact gain is positive, but it need not be large relative to the decrease in the side sum. Separating the area and shape terms makes this visible. For integers $r,s\ge3$, write $R_{r,s}=\sum_{k=1}^{s-1}\log(1-e^{-r \omega(k/s)})$.

\begin{lemma}\label{lem:torus-normal-form}
For $r,s\ge3$,
\[
  \log\tau(\Cycle_r\square\Cycle_s)
  =
  rS_s+\log(r/s)+2R_{r,s}.
\]
\end{lemma}

\begin{proof}
For $z=\omega(k/s)$, expand each factor in the torus formula of \Cref{prop:two-dimensional-formulas}:
\[
  \log\frac{\sinh(rz/2)}{\sinh(z/2)}
  =
  \frac{r-1}{2}z+\log(1-e^{-rz})-\log(1-e^{-z}).
\]
The denominator product is
\[
  \prod_{k=1}^{s-1}(1-e^{-\omega(k/s)})
  =
  e^{-S_s/2}\prod_{k=1}^{s-1}2\sinh\frac{\omega(k/s)}2
  =
  e^{-S_s/2}s,
\]
by \Cref{lem:oned-products}. Substituting this denominator leaves $rS_s+\log(r/s)+2R_{r,s}$.
\end{proof}

For $\rho>0$, define
\[
  \eta(i\rho)=e^{-\pi\rho/12}
  \prod_{\ell=1}^\infty(1-e^{-2\pi\rho\ell})
\]
and
\[
  \mathcal T(\rho)=
  \log\rho+4\log\eta(i\rho)
  =
  \log\rho-\frac{\pi\rho}{3}
  +4\sum_{\ell=1}^{\infty}\log(1-e^{-2\pi\rho\ell}).
\]

\begin{proposition}\label{prop:torus-bounded-aspect}
Uniformly when $\rho=r/s$ remains in a compact subset of $(0,\infty)$ and $r,s\to\infty$,
\[
  \log\tau(\Cycle_r\square\Cycle_s)
  =
  Irs+\mathcal T(\rho)+o(1).
\]
Moreover, the eta modular identity $\eta(i/\rho)=\sqrt{\rho}\,\eta(i\rho)$ gives
\[
  \mathcal T(\rho)=\mathcal T(1/\rho).
\]
\end{proposition}

\begin{proof}
By \Cref{lem:periodic-trap},
\[
  rS_s=Irs-\frac{\pi r}{3s}+o(1)
  =
  Irs-\frac{\pi\rho}{3}+o(1).
\]
The two ends of the residual sum supply its limiting contribution. Since $\omega(x)=2\pi x+O(x^3)$ at zero and $\omega(x)=\omega(1-x)$, for each fixed $\ell\ge1$,
\[
  r \omega(\ell/s)=2\pi\rho\ell+O(\ell^3/s^2),
  \qquad
  r \omega((s-\ell)/s)=2\pi\rho\ell+O(\ell^3/s^2),
\]
uniformly on compact $\rho$-ranges. To control the rest of the sum, $\omega(k/s)\ge c\min(k,s-k)/s$ provides a summable majorant for the logarithms throughout those ranges. Pairing the endpoints and applying dominated convergence gives
\[
  R_{r,s}=
  2\sum_{\ell=1}^{\infty}\log(1-e^{-2\pi\rho\ell})+o(1)
\]
uniformly on the stated compact ranges. Substitution into \Cref{lem:torus-normal-form} completes the expansion, and the modular identity gives its symmetry.
\end{proof}

\begin{proposition}\label{prop:torus-no-linear-constant}
There is no constant $\beta>0$ such that every nontrivial torus balancing move satisfies
\[
  \log\tau(\Cycle_a\square\Cycle_b)
  -
  \log\tau(\Cycle_A\square\Cycle_B)
  \ge
  \beta\bigl((A+B)-(a+b)\bigr).
\]
Equivalently, the best universal linear side-sum constant for tori is $0$.
\end{proposition}

\begin{proof}
We choose near-square pairs whose side-sum gap stays fixed:
\[
  (A_n,B_n)=\bigl(n^2,(n+1)^2\bigr),
  \qquad
  (a_n,b_n)=\bigl(n(n+1),n(n+1)\bigr).
\]
Then $A_nB_n=a_nb_n$, $3\le A_n<a_n=b_n<B_n$ for large $n$, and
\[
  (A_n+B_n)-(a_n+b_n)=1.
\]
The area terms in \Cref{prop:torus-bounded-aspect} cancel, while both shape terms tend to $\mathcal T(1)$. The gain therefore satisfies
\[
  \log\tau(\Cycle_{a_n}\square\Cycle_{b_n})
  -
  \log\tau(\Cycle_{A_n}\square\Cycle_{B_n})
  \to0.
\]
The side-sum drop stays equal to one, so no positive $\beta$ can work uniformly. The nonnegative comparison in \Cref{thm:torus-balancing} establishes that $\beta=0$ does.
\end{proof}

\section{Cylinders}\label{sec:cylinders}

The cylinder turns balancing into a volume-allocation problem: one direction is free and the other periodic, so neither block contains a pair to balance. At fixed area $N=rs$, increasing the circumference $s$ raises the free-boundary cost and lowers the periodic deficit. The leading penalties have orders $s$ and $N/s^2$, selecting the scale $N^{1/3}$. We first allow real height $r=N/s$ in the product formula to locate the analytic scale, then impose $s\mid N$ to control the graph optimum.

\subsection{Exact Form and Failure of Square Balancing}

For $N>0$ and integer $s\ge3$, define
\[
  F_N(s)=\log\left[
  s\prod_{k=1}^{s-1}
  \frac{\sinh((N/s)\omega(k/s))}{\sinh \omega(k/s)}
  \right],
  \qquad r=N/s.
\]
For integer $r$, \Cref{prop:two-dimensional-formulas} identifies this objective as $F_N(s)=\log\tau(\Path_r\square\Cycle_s)$. The formula extends it to real $r>0$; throughout the optimization, $3\le s\le N$ and hence $r\ge1$. Write
\[
  W_{r,s}=\sum_{k=1}^{s-1}\log(1-e^{-2r \omega(k/s)}).
\]

\begin{lemma}\label{lem:cylinder-normal-form}
For $N>0$ and integer $s\ge3$, with $r=N/s$,
\[
  F_N(s)=rS_s-\log\sinh(\alpha s)+W_{r,s}.
\]
\end{lemma}

\begin{proof}
Expand the hyperbolic quotient using $\sinh(rz)/\sinh z=e^{(r-1)z}(1-e^{-2rz})/(1-e^{-2z})$ to obtain
\[
  F_N(s)=\log s+(r-1)S_s+W_{r,s}
  -\sum_{k=1}^{s-1}\log(1-e^{-2\omega(k/s)}).
\]
The denominator can be evaluated exactly from the one-dimensional identities:
\[
  \prod_{k=1}^{s-1}2\sinh\frac{\omega(k/s)}2=s,
  \qquad
  \prod_{k=1}^{s-1}2\cosh\frac{\omega(k/s)}2=\sinh(\alpha s).
\]
The first product follows from \Cref{lem:oned-products}. For the second, square the product and apply \Cref{lem:path-cycle-polynomials} at $x=4=2\cosh(2\alpha)-2$. Together they give
\[
  \prod_{k=1}^{s-1}(1-e^{-2\omega(k/s)})
  =e^{-S_s}\prod_{k=1}^{s-1}2\sinh \omega(k/s)
  =e^{-S_s}s\sinh(\alpha s),
\]
and the normal form follows by substitution.
\end{proof}

\begin{remark}
\label{rem:cylinder-exact-gain}
For integers $3\le s_1,s_2\le N$ and $r_i=N/s_i$, the exact gain is
\[
\begin{aligned}
  F_N(s_2)-F_N(s_1)
  ={}&N\left(\frac{S_{s_2}}{s_2}-\frac{S_{s_1}}{s_1}\right)
  -\log\frac{\sinh(\alpha s_2)}{\sinh(\alpha s_1)}\\
  &+W_{r_2,s_2}-W_{r_1,s_1}.
\end{aligned}
\]
The sampling term increases with $s$ by \Cref{lem:periodic-trap}, while the boundary term satisfies
\[
  -\log\sinh(\alpha s)
  =-\alpha s+\log2-\log(1-e^{-2\alpha s}).
\]
Thus its linear coefficient is exactly $\alpha=\arsinh(1)$; any other coefficient leaves an error of order $s$.

\end{remark}

\begin{remark}\label{rem:rectangle-boundary}
The rectangle boundary coefficient can be read from the exact relation between free and periodic products in \cite[Eqs.~(54), (57)]{IzmailianKenna2015}. At unit edge weights, for integers $r,s\ge2$, this relation is
\[
\begin{aligned}
 \log\tau(\Path_r\square\Path_s)
 ={}&\frac14\log\tau(\Cycle_{2r}\square\Cycle_{2s})
      -\frac14\log(rs)+\frac14\log2\\
 &-\frac12\log\sinh(2\alpha r)
      -\frac12\log\sinh(2\alpha s).
\end{aligned}
\]
The leading expansion holds without restricting the aspect ratio. Indeed, after interchanging the sides so that $r\ge s$, the bound $\omega(x)\ge c\min(x,1-x)$ for some absolute $c>0$ makes $R_{2r,2s}$ uniformly bounded: its absolute value is bounded by a constant multiple of the convergent series $\sum_{k\ge1}e^{-ck}$. By \Cref{lem:torus-normal-form,lem:periodic-trap}, we obtain $\log\tau(\Cycle_{2r}\square\Cycle_{2s})=4Irs+O(r/s)$. Substitution gives an error $O(r/s+\log(rs))=o(r+s)$ as both sides tend to infinity, and hence
\[
 \log\tau(\Path_r\square\Path_s)=Irs-\alpha(r+s)+o(r+s).
\]
The leading cost is $\alpha/2$ per unit lattice edge-boundary length, the same as for cylinders: these lengths are $2(r+s)$ and $2s$, respectively.

As $r,s\to\infty$ with $\rho=r/s$ in a compact subset of $(0,\infty)$, \Cref{prop:torus-bounded-aspect} instead gives the uniform refinement
\[
 \log\tau(\Path_r\square\Path_s)
 =Irs-\alpha(r+s)-\frac14\log(rs)
   +\frac14\log\rho+\log\eta(i\rho)+\frac54\log2+o(1).
\]
This is the unit-weight form of \cite[Eq.~(66)]{IzmailianKenna2015}, whose free energy is $-\log\tau$. Along the sequence in the proof of \Cref{prop:torus-no-linear-constant}, the area and logarithmic corner terms cancel exactly, both shape terms have the same limit, and the side sum drops by $1$. The rectangle gain therefore tends to $\alpha$, while the torus gain tends to $0$. No uniform rectangle coefficient larger than $\alpha$ can hold.
\end{remark}

For cylinders, closest-to-square balancing already fails at area $36$.

\begin{proposition}\label{prop:cylinder-counterexamples}
Closest-to-square does not maximize cylinder spanning-tree counts at fixed area. In particular,
\[
  \tau(\Path_9\square\Cycle_4)
  =9022397309950500
  >6009209192448000
  =\tau(\Path_6\square\Cycle_6).
\]
At area $120$ one also has
\[
  \tau(\Path_{20}\square\Cycle_6)
  >\tau(\Path_{12}\square\Cycle_{10}).
\]
\end{proposition}

\begin{proof}
Integer polynomial arithmetic gives exact certificates for the two comparisons. Set $q_1(x)=1$, $q_2(x)=x+2$, and
\[
  q_{r+1}(x)=(x+2)q_r(x)-q_{r-1}(x).
\]
Then $q_r(x)=\prod_{i=1}^{r-1}(x+\lambda_i^P(r))$. If $T_s^{\mathrm{Ch}}$ is the Chebyshev polynomial, the monic polynomial
\[
  p_s(x)=(-1)^{s-1}\frac{2-2T_s^{\mathrm{Ch}}(1-x/2)}x
  =\prod_{k=1}^{s-1}(x-\lambda_k^C(s))
\]
has the nonzero cycle eigenvalues as roots, counted with multiplicity. Evaluating the symmetric product $s\prod_{k=1}^{s-1}q_r(\lambda_k^C(s))$ from these two polynomials gives the area-$36$ counts above and
\[
\begin{split}
  \tau(\Path_{20}\square\Cycle_6)
  &=1788611877285241916171285961489325853733676245188126625000,\\
  \tau(\Path_{12}\square\Cycle_{10})
  &=506827984454602606473153918154960615493126180174922465000.
\end{split}
\]
Both comparisons are strict.
\end{proof}

\subsection{Finite-Size Corrections}

At bounded aspect ratio, the normal form separates area, boundary and shape terms. The subsequent long-cylinder estimate concerns $s\to\infty$ and $r/s\to\infty$, retaining the leading periodic correction $-\pi N/(3s^2)$ needed for circumference optimization.

\begin{proposition}\label{prop:cylinder-bounded-aspect}
Let $r,s$ be integers tending to infinity, and put $\rho=r/s$. Uniformly when $\rho$ remains in a compact subset of $(0,\infty)$,
\[
  \log\tau(\Path_r\square\Cycle_s)
  =Irs-\alpha s+\mathcal C(\rho)+o(1),
\]
where
\[
  \mathcal C(\rho)=\log2-\frac{\pi\rho}{3}
  +2\sum_{\ell=1}^{\infty}\log(1-e^{-4\pi\rho\ell}).
\]
\end{proposition}

\begin{proof}
The area and boundary terms follow directly from the normal form and \Cref{lem:periodic-trap}:
\[
  rS_s=Irs-\frac{\pi\rho}{3}+O(s^{-2}),
  \qquad
  -\log\sinh(\alpha s)=-\alpha s+\log2+o(1),
\]
uniformly on each compact $\rho$-range. The residual again comes from its two ends. For each fixed $\ell\ge1$, use $\omega(x)=2\pi x+O(x^3)$ at zero and the symmetry $\omega(x)=\omega(1-x)$ to obtain
\[
  2r \omega(\ell/s)=4\pi\rho\ell+O(\ell^3/s^2),
  \qquad
  2r \omega((s-\ell)/s)=4\pi\rho\ell+O(\ell^3/s^2).
\]
The bound $\omega(k/s)\ge c\min(k,s-k)/s$ supplies a summable majorant for the logarithms, uniformly when $\rho$ is bounded away from zero. Pairing the endpoints permits dominated convergence:
\[
  W_{r,s}
  =2\sum_{\ell=1}^{\infty}\log(1-e^{-4\pi\rho\ell})+o(1)
\]
uniformly on the stated compact ranges. Adding this residual to the area and boundary terms proves the expansion.
\end{proof}

Equivalently, $\mathcal C(\rho)=\log2+2\log\eta(2i\rho)$. In the free-energy convention $F=-\log Z$, with $Z=\tau(\Path_r\square\Cycle_s)$, the expansion becomes
\[
  F=-Irs+\alpha s-\log2-2\log\eta(2i\rho)+o(1).
\]
This agrees with the isotropic normalization of Izmailian and Kenna~\cite{IzmailianKenna2015}: their periodic length is $M=s$, free length is $N=r$, and anisotropy parameter is $z=1$. There is one surface term, with no logarithmic corner contribution.

\begin{lemma}\label{lem:cylinder-residual}
There is an absolute constant $c>0$ such that
\[
  W_{r,s}=O(e^{-cr/s})
\]
uniformly for real $r>0$, integer $s\ge3$, and $r/s\ge1$.
\end{lemma}

\begin{proof}
For some absolute $\kappa>0$,
\[
  \omega(k/s)\ge\kappa\frac{\min(k,s-k)}s.
\]
This gives $e^{-2r \omega(k/s)}\le e^{-2\kappa}<1$, keeping every term in a range where $-\log(1-u)\le Cu$ holds uniformly. Each value of $\min(k,s-k)$ occurs at most twice, so
\[
  |W_{r,s}|\le C\sum_{\ell\ge1}e^{-2\kappa(r/s)\ell}
  =O(e^{-cr/s}).
\]
\end{proof}

\begin{proposition}\label{prop:cylinder-expansion}
Uniformly for integer $s\to\infty$ and real $N>0$ with $N/s^2\to\infty$,
\[
\begin{aligned}
  F_N(s)={}&IN-\alpha s-\frac{\pi N}{3s^2}+\log2\\
  &+O(N/s^4)+O(e^{-2\alpha s})+O(e^{-cN/s^2}).
\end{aligned}
\]
\end{proposition}

\begin{proof}
In the normal form, \Cref{lem:periodic-trap} gives
\[
  rS_s=IN-\frac{\pi N}{3s^2}+O(N/s^4).
\]
The boundary identity gives $-\log\sinh(\alpha s)=-\alpha s+\log2+O(e^{-2\alpha s})$, and \Cref{lem:cylinder-residual} gives the remaining error.
\end{proof}

\subsection{Optimizing the Circumference}

The competition is captured by the leading correction and its unique real maximizer:
\[
  \Phi_N(s)=-\alpha s-\frac{\pi N}{3s^2},
  \qquad
  s_0=\left(\frac{2\pi}{3\log(1+\sqrt2)}\right)^{1/3}N^{1/3}.
\]
\Cref{fig:cylinder-correction} shows how the two penalties compete and where the available divisors lie. We use the exact normal form to exclude distant competitors before applying the local expansion to a maximizer. Curvature near $s_0$ then bounds the displacement due to integrality and divisibility.

\begin{figure}[tbp]
\centering
\begin{tikzpicture}
\begin{axis}[
  width=.96\linewidth,height=6.1cm,
  xmin=19,xmax=115,ymin=-190,ymax=8,
  xlabel={Circumference $s$},ylabel={$\Phi_N(s)$ and its terms},
  title={$N=60{,}000$},
  title style={font=\small},label style={font=\small},
  tick label style={font=\footnotesize},
  axis lines=left,ymajorgrids,grid style={gray!18},
  legend style={draw=none,fill=white,font=\scriptsize,
    at={(.98,.035)},anchor=south east,cells={anchor=west}},
  clip=true]
\addplot[NavyBlue,thick,domain=19:115,samples=181]
  {-0.881373587019543*x-62831.8530717959/(x*x)};
\addlegendentry{$\Phi_N(s)$}
\addplot[gray,dashed,domain=19:115,samples=181]
  {-62831.8530717959/(x*x)};
\addlegendentry{Periodic: $-\pi N/(3s^2)$}
\addplot[BurntOrange!80!black,dashdotted,domain=19:115,samples=2]
  {-0.881373587019543*x};
\addlegendentry{Free boundary: $-\alpha s$}
\addplot[only marks,mark=*,mark size=1.5pt,NavyBlue,forget plot]
  coordinates {(20,-174.70710442) (24,-130.23604434) (25,-122.56530459) (30,-96.25437769) (32,-89.56318630) (40,-74.52485165) (48,-69.57670174) (50,-69.20142058) (60,-70.33570774) (75,-77.27312624) (80,-80.32736400) (96,-91.42955674) (100,-94.42054401)};
\draw[NavyBlue!65,dotted] (axis cs:52.2416169948,-190)--(axis cs:52.2416169948,-69.0665720436);
\addplot[only marks,mark=o,mark size=3pt,thick,NavyBlue,forget plot]
  coordinates {(52.2416169948,-69.0665720436)};
\node[font=\small,text=NavyBlue,anchor=west] at (axis cs:63,-48)
  {$s_0\approx52.24$};
\draw[NavyBlue,thin] (axis cs:62,-52)--(axis cs:52.2416169948,-69.0665720436);
\addplot[only marks,mark=*,mark size=2pt,BurntOrange!80!black,forget plot]
  coordinates {(50,-69.2014205797)};
\node[font=\small,text=BurntOrange!80!black,anchor=east] at (axis cs:46,-96)
  {$u_N=50$};
\draw[BurntOrange!80!black,thin] (axis cs:45,-91)--(axis cs:50,-69.2014205797);
\end{axis}
\end{tikzpicture}
\caption{The leading correction $\Phi_N(s)=-\alpha s-\pi N/(3s^2)$ at $N=60{,}000$. Dashed curves show the separate contributions. The open circle marks the continuous maximum $s_0$; filled dots mark divisors of $N$. The nearby divisor $u_N=50$ illustrates the hypothesis of \Cref{thm:cylinder-divisor-rich}.}
\label{fig:cylinder-correction}
\end{figure}

\begin{theorem}\label{thm:cylinder-real-height}
Let $N\to\infty$, and let $s_N^*$ be any integer maximizer of $F_N(s)$ over $3\le s\le N$, with $r=N/s$ allowed to be real. Then
\[
  s_N^*=\left(\frac{2\pi}{3\arsinh(1)}\right)^{1/3}N^{1/3}+O(1).
\]
\end{theorem}

\begin{theorem}[Divisor stability]\label{thm:cylinder-divisor-rich}
Let $N\to\infty$ through integers admitting a divisor $u_N\mid N$, $u_N\ge3$, such that $u_N/s_0\to1$. Every circumference $s_{\mathrm{opt}}$ maximizing $\tau(\Path_{N/s}\square\Cycle_s)$ over divisors $s\mid N$, $s\ge3$, satisfies
\[
  |s_{\mathrm{opt}}-s_0|=O(|u_N-s_0|+1).
\]
Consequently $s_{\mathrm{opt}}\sim s_0$; if $u_N=s_0+O(1)$, then $s_{\mathrm{opt}}=s_0+O(1)$.
\end{theorem}

\begin{proof}[Proof of \Cref{thm:cylinder-real-height,thm:cylinder-divisor-rich}]
We prove both statements by comparison with a nearby admissible circumference. For \Cref{thm:cylinder-real-height}, take an integer $v_N$ nearest to $s_0$; for \Cref{thm:cylinder-divisor-rich}, take $v_N=u_N$. In both cases $v_N/s_0\to1$, and the choice is admissible for all sufficiently large $N$. Write $L=N^{1/3}$.

Two bounds locate maximizers on scale $L$ by excluding the extremes of the full domain. The inequalities $W_{r,s}\le0$ and $S_s/s<I$ in the normal form give
\[
  F_N(s)\le IN-\log\sinh(\alpha s)
  \le IN-\alpha s+C_0
  \qquad(s\ge3),
\]
where $C_0=\log2-\log(1-e^{-6\alpha})$. For the other bound, combine the strict inequality $S_s/s<I$ with its expansion in \Cref{lem:periodic-trap}. There is an absolute $c_0>0$ such that
\[
  \frac{S_s}{s}\le I-\frac{c_0}{s^2}
  \qquad(s\ge3).
\]
Indeed, $s^2(I-S_s/s)\to\pi/3$, and each of the finitely many remaining values is positive. Since $\sinh(\alpha s)>1$ for $s\ge3$, the normal form also gives
\[
  F_N(s)\le IN-\frac{c_0N}{s^2}.
\]

The competitor has $v_N\asymp L$, and \Cref{prop:cylinder-expansion} gives $F_N(v_N)\ge IN-C_1L$ for a fixed $C_1$. Choose fixed $M$ large enough that $\alpha M>C_1+1$ and fixed $\varepsilon>0$ small enough that $c_0\varepsilon^{-2}>C_1+1$. The first global bound excludes $s\ge ML$ from maximizing; the second excludes $s\le\varepsilon L$. Every maximizer therefore lies in $[\varepsilon L,ML]$. Enlarge this fixed window if necessary to include both $s_0$ and $v_N$.

Within this window, \Cref{prop:cylinder-expansion} supplies the uniform approximation
\[
  F_N(s)=IN+\log2+\Phi_N(s)+E_N(s),
  \qquad |E_N(s)|\le K/L.
\]
The leading correction satisfies $\Phi_N'(s_0)=0$, and its curvature is uniformly controlled:
\[
  \frac{c_2}{L}\le-\Phi_N''(s)=\frac{2\pi N}{s^4}
  \le\frac{C_2}{L}
\]
throughout the interval, for fixed positive $c_2,C_2$. Integrating between $s_0$ and $s$ turns curvature into a quadratic loss:
\[
  \frac{c_2}{2L}(s-s_0)^2
  \le\Phi_N(s_0)-\Phi_N(s)
  \le\frac{C_2}{2L}(s-s_0)^2.
\]
For an admissible maximizer $s_*$, the comparison $F_N(s_*)\ge F_N(v_N)$ bounds that loss by the competitor's loss and the approximation error:
\[
\begin{aligned}
  \frac{c_2}{2L}(s_*-s_0)^2
  &\le\Phi_N(s_0)-\Phi_N(s_*)\\
  &\le\Phi_N(s_0)-\Phi_N(v_N)+\frac{2K}{L}\\
  &\le\frac{C_2}{2L}(v_N-s_0)^2+\frac{2K}{L}.
\end{aligned}
\]
It follows that $|s_*-s_0|=O(|v_N-s_0|+1)$. In the relaxed problem, $|v_N-s_0|\le1/2$ gives \Cref{thm:cylinder-real-height}. In the divisor problem, $v_N=u_N$ gives the quantitative bound; the assumptions $u_N/s_0\to1$ and, respectively, $u_N=s_0+O(1)$ yield its two stated consequences.
\end{proof}

\begin{remark}\label{rem:cylinder-arithmetic}
The real-height relaxation in \Cref{thm:cylinder-real-height} keeps $s$ integer but drops the divisibility condition $s\mid N$. The scale $s_0$ maximizes the leading correction $\Phi_N$, or equivalently minimizes the loss $-\Phi_N$. A graph circumference must instead be a divisor of $N$ with $s\ge3$, and such divisors need not lie near $s_0$. For prime $N\ge3$, the only choice is $s=N$, giving $P_1\square C_N\cong C_N$. The hypothesis in \Cref{thm:cylinder-divisor-rich} requires one divisor $u_N=s_0+o(s_0)$; it then bounds every graph maximizer's displacement by $O(|u_N-s_0|+1)$. It does not assert that a nearby divisor exists for every $N$ or that the nearest divisor is always the exact maximizer.
\end{remark}

\section{Concluding Remarks}\label{sec:discussion}

Logarithmic balancing controls same-type Cartesian blocks at fixed volume. Stronger quantitative bounds in higher dimensions are a further direction; leaving the product class or allowing volume to move between boundary types introduces different extremal problems.

\paragraph{Beyond rectangular lattice subgraphs.}
Let $\mathcal L_d$ be the nearest-neighbor graph on $\Z^d$. The planar case $d=2$ is settled: Zhang~\cite{Zhang2026SquareMaximality} proves that the $n\times n$ square uniquely maximizes the spanning-tree count among connected induced square-lattice subgraphs with $n^2$ vertices, up to translation. A Lean formalization accompanies that result in Palomar. The higher-dimensional extension is the following conjecture.

\begin{conjecture}\label{conj:lattice-product-extremal}
Let $d\ge3$ and $n\ge1$. If $S\subset\Z^d$ has $|S|=n^d$ and $\mathcal L_d[S]$ is connected, then
\[
  \tau(\mathcal L_d[S])\le\tau(P_n^{\square d}),
\]
with equality only for a translate of the equal-sided lattice box.
\end{conjecture}

The perfect-power theorem supplies the product case; extending it to arbitrary connected induced subgraphs in dimensions $d\ge3$ requires control of shapes beyond boxes. At other volumes, without a perfect-power constraint, an optimal induced subgraph need not be a Cartesian product at all.

\paragraph{Quantitative balancing in higher dimensions.}
The rectangle lower bound is proportional to the drop in half the lattice edge-boundary length. For a product box in $\mathcal L_d$, the corresponding quantity is
\[
  \sigma(\mathbf n)=\sum_{i=1}^d\prod_{j\ne i}n_j
  =(n_1\cdots n_d)\sum_{i=1}^d\frac1{n_i}.
\]
There are $2\sigma(\mathbf n)$ lattice edges leaving the box, including when some side lengths equal one. Strict convexity of $x\mapsto e^{-x}$ and Karamata's inequality give $\sigma(\mathbf n)\ge\sigma(\mathbf m)$ under logarithmic majorization, with equality only for permutation tuples.

In dimensions $d\ge3$, it is natural to ask whether
\[
  \log\tau(\Path_{\mathbf m})-\log\tau(\Path_{\mathbf n})
  \ge c_d\bigl(\sigma(\mathbf n)-\sigma(\mathbf m)\bigr)
\]
holds with $c_d>0$ depending only on $d$, for all equal-volume positive integer tuples $\mathbf n$ log-majorizing $\mathbf m$. In dimension two, \Cref{thm:rectangle-balancing} gives $c_2=\alpha$. The present heat-trace argument establishes the sign of the gain but no such positive uniform coefficient in higher dimensions. For two-dimensional tori, the analogous side-sum coefficient is zero by \Cref{prop:torus-no-linear-constant}.

\paragraph{Periodic identifications change the comparison class.}
The standard torus $C_n^{\square d}$ ceases to be a universal benchmark once arbitrary periodic identifications are allowed. Consider
\[
  \Lambda=\langle(9,0),(-3,1)\rangle\subset\Z^2.
\]
The map $(a,b)\mapsto a+3b\pmod9$ identifies $\Z^2/\Lambda$ with $\Z/9\Z$, and the simple square-lattice quotient joins each vertex $i$ to $i\pm1$ and $i\pm3$. This connected graph has nine vertices, yet
\[
  \tau(\mathcal L_2/\Lambda)=12321>11664=\tau(C_3\square C_3).
\]
The change of periodic identification is shown in \Cref{fig:periodic-quotient}.

\begin{figure}[H]
\centering
\begin{tikzpicture}[x=1.15cm,y=1.05cm,shared/.style={gray!65,line width=.45pt},changed/.style={NavyBlue,line width=1pt},vertex/.style={circle,draw=black!80,fill=white,inner sep=1.65pt,font=\footnotesize}]
\begin{scope}[shift={(0,0)}]
\coordinate (std0) at (0,2);
\coordinate (std1) at (1,2);
\coordinate (std2) at (2,2);
\coordinate (std3) at (0,1);
\coordinate (std4) at (1,1);
\coordinate (std5) at (2,1);
\coordinate (std6) at (0,0);
\coordinate (std7) at (1,0);
\coordinate (std8) at (2,0);
\draw[shared] (std0)--(std1);
\draw[shared] (std0)--(std3);
\draw[shared] (std1)--(std2);
\draw[shared] (std1)--(std4);
\draw[shared] (std2)--(std5);
\draw[shared] (std3)--(std4);
\draw[shared] (std3)--(std6);
\draw[shared] (std4)--(std5);
\draw[shared] (std4)--(std7);
\draw[shared] (std5)--(std8);
\draw[shared] (std6)--(std7);
\draw[shared] (std7)--(std8);
\draw[shared,dashed] (std0) to[bend left=20] (std6);
\draw[shared,dashed] (std1) to[bend left=20] (std7);
\draw[shared,dashed] (std2) to[bend left=20] (std8);
\draw[changed] (std2) to[bend right=35] (std0);
\draw[changed] (std5) to[bend right=35] (std3);
\draw[changed] (std8) to[bend right=35] (std6);
\node[vertex] at (std0) {0};
\node[vertex] at (std1) {1};
\node[vertex] at (std2) {2};
\node[vertex] at (std3) {3};
\node[vertex] at (std4) {4};
\node[vertex] at (std5) {5};
\node[vertex] at (std6) {6};
\node[vertex] at (std7) {7};
\node[vertex] at (std8) {8};
\node[font=\small] at (1,3.22) {$C_3\square C_3$};
\node[font=\small,text=NavyBlue] at (1,-.52) {$\tau=11{,}664$};
\end{scope}
\begin{scope}[shift={(5.4,0)}]
\coordinate (skew0) at (0,2);
\coordinate (skew1) at (1,2);
\coordinate (skew2) at (2,2);
\coordinate (skew3) at (0,1);
\coordinate (skew4) at (1,1);
\coordinate (skew5) at (2,1);
\coordinate (skew6) at (0,0);
\coordinate (skew7) at (1,0);
\coordinate (skew8) at (2,0);
\draw[shared] (skew0)--(skew1);
\draw[shared] (skew0)--(skew3);
\draw[shared] (skew1)--(skew2);
\draw[shared] (skew1)--(skew4);
\draw[shared] (skew2)--(skew5);
\draw[shared] (skew3)--(skew4);
\draw[shared] (skew3)--(skew6);
\draw[shared] (skew4)--(skew5);
\draw[shared] (skew4)--(skew7);
\draw[shared] (skew5)--(skew8);
\draw[shared] (skew6)--(skew7);
\draw[shared] (skew7)--(skew8);
\draw[shared,dashed] (skew0) to[bend left=20] (skew6);
\draw[shared,dashed] (skew1) to[bend left=20] (skew7);
\draw[shared,dashed] (skew2) to[bend left=20] (skew8);
\draw[changed] (skew2) to[bend right=10] (skew3);
\draw[changed] (skew5) to[bend right=10] (skew6);
\draw[changed] (skew8) .. controls (3.8,0) and (1.8,4.3) .. (skew0);
\node[vertex] at (skew0) {0};
\node[vertex] at (skew1) {1};
\node[vertex] at (skew2) {2};
\node[vertex] at (skew3) {3};
\node[vertex] at (skew4) {4};
\node[vertex] at (skew5) {5};
\node[vertex] at (skew6) {6};
\node[vertex] at (skew7) {7};
\node[vertex] at (skew8) {8};
\node[font=\small] at (1,3.22) {Skew quotient $\mathcal L_2/\Lambda$};
\node[font=\small,text=NavyBlue] at (1,-.52) {$\tau=12{,}321$};
\end{scope}
\node[font=\footnotesize,align=center] at (3.75,-1.03) {Gray: shared edges. Blue: changed edges.};
\end{tikzpicture}
\caption{Changing three periodic edges increases the spanning-tree count. Gray edges are shared; blue edges differ. Both graphs have nine vertices and degree four. Only marked circles are vertices.}
\label{fig:periodic-quotient}
\end{figure}
The counts can be read directly from the spectra. The quotient has Laplacian eigenvalues $4-2\cos(2\pi k/9)-2\cos(2\pi k/3)$. Its modes $k=3,6$ give $3$; the remaining six nonzero modes give $5-r$, with each root of $r^3-3r+1$ occurring twice. Their product is $111^2$, so the Matrix--Tree Theorem gives $3^2\cdot111^2/9=12321$. For the standard torus, positive eigenvalues $3$ and $6$ each occur four times, giving $3^4 6^4/9=11664$. The Cartesian-product restriction is therefore essential to our torus comparison.

For large quotients, determinant asymptotics connect this finite problem to the height of a flat torus. Chinta, Jorgenson, and Karlsson~\cite[Theorem~1.1]{ChintaJorgensonKarlsson2012} show that, as the index $N\to\infty$ and the period basis matrices divided by $\sqrt N$ converge to a determinant-one matrix, the shape-dependent term in $\log\tau$ is the logarithm of the limiting unit-area flat torus's zeta-regularized determinant. Its negative logarithm is the torus height, so maximizing the determinant minimizes that height. The maximizing torus is hexagonal~\cite[Corollary~1(b)]{OsgoodPhillipsSarnak1988}; its period lattice also underlies the densest planar packing of equal circles, a connection discussed in~\cite{ChintaJorgensonKarlsson2012}.

The determinant limit also yields a comparison beyond rectangular tori. Take two sequences with the same growing indices: if the normalized period lattices of one converge to the hexagonal shape, while those of the other stay in a compact family of nondegenerate shapes separated from it, the first sequence eventually has more spanning trees~\cite[Corollary~1.2]{ChintaJorgensonKarlsson2012}. The hexagonal shape belongs to the period lattice; the finite graphs still have square-lattice adjacency. At a prescribed finite index, the available integer lattices and finite-size corrections remain to be compared.

\paragraph{Mixed products.}
When the number of factors and the volume of each boundary type are fixed separately, \Cref{thm:log-majorization} balances both blocks. For fixed integers $p,q\ge1$, $a\ge1$, and $b\ge3$, the graph $P_a^{\square p}\square C_b^{\square q}$ uniquely maximizes the count among products of $p$ path factors of total volume $a^p$ and $q$ cycle factors of total volume $b^q$. Fixing only the overall volume leaves an additional allocation problem. The heat trace contains two different factor functions, so one Karamata comparison does not determine that allocation; the cylinder's cube-root optimum illustrates the further optimization required.

\section*{Acknowledgements}

The core results in this paper were obtained by the author in early 2026. Gemini 3.1, GPT-5.4, GPT-5.5, and GPT-5.6 assisted with language editing, formatting, literature research, and preparing plots and diagrams. These tools also helped the author identify and better understand connections between the results and classical work, and assisted with selected mathematical checks, including checks of the Laplacian product formulas for cylinders and tori. GPT-5.5 also assisted with additional verification, including the computation of selected spanning-tree counts. Codex assisted with revising the manuscript and helped identify and explore some extensions of the arguments. The author takes full responsibility for the content of the paper.

\bibliographystyle{alpha}
\bibliography{ref}
\end{document}